\documentclass[11pt, reqno]{amsart}

\usepackage{amsmath, amsfonts, amssymb, amsbsy, bigstrut, graphicx, enumerate,  upref, longtable, comment, booktabs, array, caption, subcaption}

\usepackage{pstricks}

\allowdisplaybreaks

\usepackage[breaklinks]{hyperref}

\usepackage[numbers, sort&compress]{natbib}

\newtheorem{thm}{Theorem}
\newtheorem{lmm}{Lemma}
\newtheorem{cor}{Corollary}

\newtheorem{defn}{Definition}

\theoremstyle{definition}

\newcommand{\md}{\mathcal{D}}

\newcommand{\cp}{\mathcal{P}}
\newcommand{\cq}{\mathcal{Q}}

\newcommand{\rr}{\mathbb{R}}

\newcommand{\tr}{\operatorname{Tr}}

\newcommand{\ve}{\varepsilon}

\newcommand{\rank}{\operatorname{rank}}

\usepackage[utf8]{inputenc}
\usepackage{amsmath}
\usepackage{amsfonts}
\usepackage{bbm}
\usepackage{mathtools}
\usepackage{tikz}
\usepackage{amsthm}
\usepackage[english]{babel}
\usepackage{fancyhdr}
\usepackage{amssymb}
\usepackage{enumitem}
\usepackage{qtree}
\usepackage{mathrsfs}

\renewcommand{\bar}{\overline}

\renewcommand{\tilde}{\widetilde}
\renewcommand{\hat}{\widehat}

\begin{document}

\title{A deterministic theory of low rank matrix completion}
\author{Sourav Chatterjee}
\address{Department of Statistics, Stanford University, Sequoia Hall, 390 Serra Mall, Stanford, CA 94305}
\email{souravc@stanford.edu}
\thanks{Research partially supported by NSF grant DMS-1855484}

\begin{abstract}
The problem of completing a large low rank matrix using a subset of revealed entries has received much attention in the last ten years. The main result of this paper gives a necessary and sufficient condition, stated  in the language of graph limit theory, for a sequence of matrix completion problems with arbitrary missing patterns to be asymptotically solvable. It is then shown that a small modification of the Cand\`es--Recht nuclear norm minimization algorithm provides the required asymptotic solution whenever the sequence of problems is asymptotically solvable. The theory is fully deterministic, with no assumption of randomness. A number of open questions are listed. 
\end{abstract}

\maketitle

\section{Introduction}
The problem of reconstructing a large low rank matrix from a subset of revealed entries has attracted widespread attention in the statistics and machine learning literatures in the last ten years. For a recent survey of this vast body of work, see~\cite{nks19}. Notice that the problem itself is a problem in linear algebra, with nothing  random in it. However, matrix completion in classical linear algebra is restricted to matrices with special structure, such as positive definite matrices~\cite{johnson90}. 

In the literature on low rank matrix completion, randomness enters into the picture through the assumption that the set of missing entries is random. In most papers, the randomness is uniform over all subsets of a given size. This assumption, while unrealistic, allows researchers to prove many beautiful theorems. There are a handful of papers that strive to work with deterministic missing patterns or missing patterns that depend on the matrix, using spectral gap conditions~\cite{hss14, bj14},  rigidity theory~\cite{sc10}, algebraic geometry~\cite{ktt15} and other methods~\cite{cbsw, pbn16, ls13, sbj18}. These are discussed in some detail in Section \ref{litsec}.

However, a complete characterization of missing patterns that allow approximate completion of large low rank matrices has remained an open question. The aim of this paper is to give such a characterization. Here `approximate completion' means that the missing entries are required to be recovered approximately and not exactly (a precise definition is given later). The analogous criterion for exact recovery is left as an open question. Our result is an asymptotic statement involving limits; proving a non-asymptotic version of the result is also left as an open question.

Right away, it is important to note that not all patterns of revealed entries allow low rank matrix completion (even in an approximate sense), even if a substantial fraction of entries are revealed. For example, if we have a large square matrix of order $n$, and only the top $n/2$ rows are revealed, the matrix cannot be completed even if it is known to have rank $1$ (see Figure \ref{mat1}). When we say `cannot be completed', what we really mean is that there multiple very different ways to complete, even under the low rank assumption. This means that any particular completion cannot be a reliable estimate of the true matrix. 

\begin{figure}[t!]
\begin{center}
\begin{tikzpicture}[scale = 1]
\draw[thick] (0,0) to (0,4);
\draw[thick] (4,0) to (4,4);
\draw[thick] (0,0.01) to (0.2,0.01);
\draw[thick] (0,3.99) to (0.2,3.99);
\draw[thick] (4,0.01) to (3.8,0.01);
\draw[thick] (4,3.99) to (3.8,3.99);
\draw[lightgray, fill] (.1,2) rectangle (3.9,3.9);
\node at (2,1) {Missing};
\node at (2,3) {Available};
\end{tikzpicture}
\caption{A pattern of missing entries that cannot be completed (even approximately) even if the rank is known to be small. \label{mat1}}
\end{center}
\end{figure}
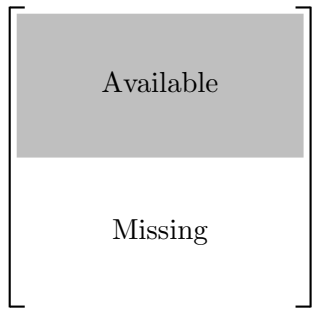

This example suggests that the set of revealed entries has to be in some sense `dense' in the set of all entries for the matrix to be recoverable. However, one has to be cautious about this intuition. Consider a second counterexample: Let $n$ be even, and consider an $n\times n$ matrix whose $(i,j)^{\textup{th}}$ entry is revealed if and only if $i$ and $j$ have the same parity (that is, both even or both odd). This set of revealed entries looks sufficiently `dense' (see Figure~\ref{mat2}). Yet, we will now argue that recovery is not possible even if the rank of the matrix is as small as three.  

To see this, note that the rows and columns can be relabeled such that the even numbered rows and columns in the original matrix are renumbered from $1$ to $n/2$ and the odd numbered rows and columns are renumbered from $n/2+1$ to $n$. Then in this new arrangement of rows and columns, the $(i,j)^{\textup{th}}$ entry is revealed if and only if either both $i$ and $j$ are between $1$ and $n/2$, or both $i$ and $j$ are between $n/2+1$ and $n$. In other words, the matrix is a $2\times 2$ block matrix with blocks of order $n/2\times n/2$, where only the top-left and bottom-right blocks are revealed (again, see Figure \ref{mat2}). Clearly, the other two blocks cannot be recovered using this information if the rank is three  or higher.

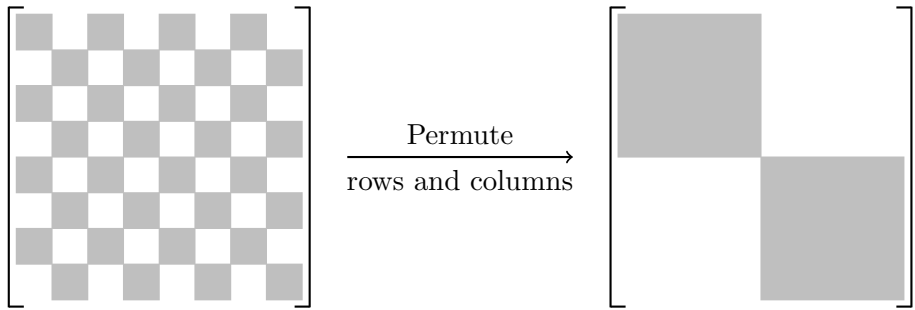
\begin{figure}[t!]
\begin{center}
\begin{tikzpicture}[scale = 1]
\draw[thick] (0,0) to (0,4);
\draw[thick] (4,0) to (4,4);
\draw[thick] (0,0.01) to (0.2,0.01);
\draw[thick] (0,3.99) to (0.2,3.99);
\draw[thick] (4,0.01) to (3.8,0.01);
\draw[thick] (4,3.99) to (3.8,3.99);
\draw[thick] (8,0) to (8,4);
\draw[thick] (12,0) to (12,4);
\draw[thick] (8,0.01) to (8.2,0.01);
\draw[thick] (8,3.99) to (8.2,3.99);
\draw[thick] (12,0.01) to (11.8,0.01);
\draw[thick] (12,3.99) to (11.8,3.99);
\foreach \i in {0,...,3} {
  \foreach \j in {0,...,3} { 
     \draw[lightgray, fill] ({.1 + .475*(2*\i+1)}, {.1 + .475*2*\j}) rectangle ({.1 + .475*(2*\i+2)}, {.1 + .475*(2*\j+1)});
  }
}
\foreach \i in {0,...,3} {
  \foreach \j in {0,...,3} { 
     \draw[lightgray, fill] ({.1 + .475*(2*\i)}, {.1 + .475*(2*\j+1)}) rectangle ({.1 + .475*(2*\i+1)}, {.1 + .475*(2*\j+2)});
  }
}
\draw[lightgray, fill] (8.1,2) rectangle (10,3.9);
\draw[lightgray, fill] (10,.1) rectangle (11.9,2);
\draw[thick, ->] (4.5,2) to (7.5,2);
\node at (6,2.3) {Permute};
\node at (6,1.7) {rows and columns};
\end{tikzpicture}
\caption{A pattern of missing entries that cannot be completed. Shaded regions denote available entries and white regions are missing. Although the pattern of available entries on the left looks `dense', permuting rows and columns in a particular way gives the pattern on the right, which is clearly not completable (even approximately) even if the rank is known to be small. \label{mat2}}
\end{center}
\end{figure}

The problem with the above counterexample is that the rows and columns could  be relabeled so that the pattern of revealed entries is no longer `dense'. This suggests that for recoverability of low rank matrices, it is necessary that the pattern of revealed entries {\it remains `dense' under any  relabeling of rows and columns.}   

It turns out that this condition is also sufficient. This is the main theorem of this paper (Theorem \ref{mainthm}). The precise statement is given in the language of graph limit theory~\cite{lovaszbook}.  It is then proved that a modification of a popular method of low rank matrix completion by nuclear norm minimization~\cite{candesrecht09, candestao10, candesplan10} succeeds in approximately recovering the full matrix whenever the above condition holds (Theorem \ref{convthm}). In other words, this algorithm does the job whenever the job is doable. 

The modification is as follows. The usual Cand\`es--Recht algorithm finds the matrix with minimum nuclear norm among all matrices that agree with the unknown matrix on the set of revealed entries. In the modification, we assume that an upper bound $L$ on the magnitudes of the entries of the unknown matrix is known to the user (which is usually true), and then find the matrix that minimizes the nuclear norm subject to the usual constraint, plus the constraint that the magnitudes of all entries are bounded by $L$. Note that this is still a convex optimization problem, just like the original algorithm. 

The rest of the paper is organized as follows. Some necessary notations are introduced in Section \ref{notsec}. Several definitions needed for stating our results in the language of graph limit theory are given in Section \ref{defsec}. The  main results are presented in Section \ref{resultsec}. A brief discussion of the existing literature on matrix completion with non-uniform missing patterns is given in Section \ref{litsec}. Some open problems are stated in Section \ref{opensec}. The remaining sections are devoted to proofs.

\section{Notations}\label{notsec}
All our matrices will have real entries. We will denote the $(i,j)^{\textup{th}}$ entry of a matrix $A$ by $a_{ij}$, of $B$ by $b_{ij}$, and so on. The transpose of a matrix $A$ will be denoted by $A^T$, and the trace by $\tr(A)$, and the rank by $\rank(A)$. Vectors will be treated as matrices with one column.

Let $A$ be an $m\times n$ matrix. We will have occasions to use the following matrix norms. The {\it Frobenius norm} of $A$ is defined as
\[
\|A\|_F := \biggl(\sum_{i=1}^m \sum_{j=1}^n a_{ij}^2\biggr)^{1/2}.
\]
More frequently, we will use the following {\it averaged} version of the Frobenius norm:
\[
\|A\|_{\bar{F}} := \frac{\|A\|_F}{\sqrt{mn}} = \biggl(\frac{1}{mn}\sum_{i=1}^m \sum_{j=1}^n a_{ij}^2\biggr)^{1/2}.
\]
For us, the average Frobenius norm will be more useful than the usual Frobenius norm because it is a measure of the size of a typical entry of $A$. 

If $\sigma_1,\ldots, \sigma_r$ are the non-zero singular values of $A$, the {\it nuclear norm} of $A$ is defined as
\[
\|A\|_* := \sum_{i=1}^r \sigma_i. 
\]
The {\it $\ell^\infty$ norm} of $A$ is simply
\[
\|A\|_\infty:= \max_{i,j} |a_{ij}|. 
\]
We will also use a somewhat non-standard matrix norm, called the {\it cut norm}, defined as
\[
\|A\|_\Box := \frac{1}{mn}\max\{|x^T A y| :x\in \rr^m,\, y\in \rr^n,\, \|x\|_\infty \le 1, \, \|y\|_\infty\le 1\}. 
\]
In the usual definition of the cut norm for matrices, the maximum is not divided by $mn$. We divide by $mn$ because it will be more convenient for us to work with this version, and also because this is the custom in graph limit theory. 

For each $k$, let $S_k$ be the group of all permutations of $\{1,\ldots, k\}$. For $\pi\in S_m$ and $\tau\in S_n$, let $A^{\pi,\tau}$ be the matrix whose $(i,j)^{\textup{th}}$ entry is $a_{\pi(i)\tau(j)}$. The cut norm is used to define the {\it cut distance} between two $m\times n$ matrices $A$ and $B$ as 
\begin{align}\label{cutdef}
\delta_\Box(A,B) := \min_{\pi\in S_m, \, \tau\in S_n} \|A^{\pi,\tau} - B\|_\Box. 
\end{align}
We will say that a matrix is a {\it binary matrix} if each of its entries is either $0$ or $1$. We will use binary matrices to denote the locations of revealed entries in matrix completion problems.  

If $A$ and $B$ are two $m\times n$ matrices, the {\it Hadamard product} of $A$ and $B$, denoted by $A\circ B$, is the $m\times n$ matrix whose $(i,j)^{\textup{th}}$ entry is $a_{ij}b_{ij}$. Hadamard products will be useful for us in the following way. If $A$ is a matrix which is partially revealed, and $P$ is a binary matrix indicating the locations of the revealed entries, then $A\circ P$ is the matrix whose entries equal the entries of $A$ wherever they are revealed, and zero elsewhere.

\section{Definitions}\label{defsec}
As mentioned earlier, certain patterns of revealed entries may not suffice for approximately recovering the full matrix, whereas other patterns may suffice. While this makes intuitive sense, we need to give a precise mathematical definition of the notion of recoverability before proceeding further with this. Roughly speaking, approximate recoverability should mean that if two low rank matrices are approximately equal on the revealed entries, they should also be approximately equal everywhere. To make this fully precise, we need to state it in terms of sequences of matrices rather than a single matrix.
\begin{defn}\label{recovdef}
Let $\{P_k\}_{k\ge 1}$ be a sequence of binary matrices. We will say that this sequence \emph{admits stable recovery of low rank matrices} if it has the following property. Take any two sequences of matrices $\{A_k\}_{k\ge 1}$ and $\{B_k\}_{k\ge 1}$, where  $A_k$ and $B_k$ have the same dimensions as $P_k$. Suppose that there are numbers $K$ and $L$ such that $\rank(A_k)$ and $\rank(B_k)$ are bounded by $K$ and $\|A_k\|_\infty$ and $\|B_k\|_\infty$ are bounded by $L$ for each $k$. Then for any $\ve>0$ there is some $\delta>0$, depending only on $\ve$, $K$ and $L$, such that if
$\limsup_{k\to\infty} \|(A_k - B_k)\circ P_k\|_{\bar{F}} \le \delta$, 
then $\limsup_{k\to\infty} \|A_k-B_k\|_{\bar{F}} \le \ve$. 
\end{defn}
The word `stable' is added in the above definition to emphasize that we only need approximate equality of the revealed entries, rather than exact equality.

To understand the essence of the above definition, it is probably helpful to revisit the counterexample mentioned earlier. For each $k$, let $P_k$ be the $k\times k$ binary matrix whose entries are $1$ in the first $[k/2]$ rows, and $0$ elsewhere. Let $A_k$ be the $k\times k$ matrix of all zeros, and $B_k$ be the $k\times k$ matrix whose entries are $0$ in the top $[k/2]$ rows and $1$ elsewhere. Then $\|A_k\|_\infty$ and $\|B_k\|_\infty$ are bounded by $1$ for all $k$, and $\rank(A_k)$ and $\rank(B_k)$ are bounded by $1$ for all $k$. Now clearly $\lim_{k\to\infty} \|(A_k - B_k)\circ P_k\|_{\bar{F}} = 0$, but a simple calculation shows that
\[
\lim_{k\to\infty} \|A_k-B_k\|_{\bar{F}} = \frac{1}{\sqrt{2}}.
\]
Thus, the sequence $\{P_k\}_{k\ge 1}$ does not admit stable recovery of low rank matrices.

To verify that a sequence $\{P_k\}_{k\ge 1}$ admits stable recovery of low rank matrices according to Definition \ref{recovdef}, one needs to verify the stated condition for all sequences $\{A_k\}_{k\ge 1}$ and $\{B_k\}_{k\ge 1}$. It would however be much more desirable to have an equivalent criterion in terms of some intrinsic property of the sequence $\{P_k\}_{k\ge 1}$. The main result of this paper gives such a criterion. To state this result, we need to introduce some further definitions. 

In graph limit theory~\cite{lovaszbook}, a {\it graphon} is a Borel measurable function from $[0,1]^2$ into $[0,1]$, which is symmetric in its arguments. Since we are dealing with matrices that need not be symmetric, we need to generalize this definition by dropping the symmetry condition.
\begin{defn}
An \emph{asymmetric graphon} is a Borel measurable function from $[0,1]^2$ into $[0,1]$.
\end{defn}
If $W$ is an asymmetric graphon and $m$ and $n$ are two positive integers, we define the $m\times n$ discrete approximation of $W$ to be the $m\times n$ matrix $W_{m,n}$, whose $(i,j)^{\textup{th}}$ entry is the average value of $W$ in the rectangle $[\frac{i-1}{m}, \frac{i}{m}]\times [\frac{j-1}{n}, \frac{j}{n}]$, that is, 
\[
mn\int_{(i-1)/m}^{i/m} \int_{(j-1)/n}^{j/n} W(x,y)dydx. 
\]
If $A$ is an $m\times n$ matrix and $W$ is an asymmetric graphon, we define the cut distance between $A$ and $W$ to be
\[
\delta_\Box(A, W) := \delta_\Box(A, W_{m,n}),
\]
where the right side is as defined in equation \eqref{cutdef}. 
\begin{defn}\label{limdef}
We will say that a sequence of matrices $\{A_k\}_{k\ge1}$ converges to an asymmetric graphon $W$ if $\delta_\Box(A_k, W)\to0$ as $k\to \infty$. 
\end{defn}
Note that the limit defined in the above sense may not be unique. The same sequence may converge to many different limits. In graph limit theory, all of these different limits are considered to be equivalent by defining an equivalence relation on the space of graphons. It is possible to do a similar thing for asymmetric graphons, but that is not needed for this paper. 
 
We will use asymmetric graphons to represent limits of binary matrices. Not every sequence has a limit, but subsequential limits always exist.
\begin{thm}\label{compactthm}
Any sequence of binary matrices with dimensions tending to infinity has a subsequence that converges to an asymmetric graphon.
\end{thm}
The above theorem is the asymmetric analog of a fundamental compactness theorem in graph limit theory~\cite[Theorem 9.23]{lovaszbook}. It is probable that the asymmetric version already exists in the literature, but since the proof is not difficult, it is presented in Section \ref{compactproof}.

\section{Results}\label{resultsec}
Our main objective is to give a necessary and sufficient condition for a sequence of binary matrices to admit stable recovery of low rank matrices. Because of Theorem \ref{compactthm}, it suffices to only consider convergent sequences. 
\begin{thm}\label{mainthm}
A sequence of binary matrices with dimensions tending to infinity and converging to an asymmetric graphon $W$ admits stable recovery of low rank matrices (in the sense of Definition \ref{recovdef}) if and only if  $W$ is nonzero almost everywhere. 
\end{thm}
To understand this result, first consider the familiar case of entries missing at random. Suppose that each entry is revealed with probability $p$, independently of each other. Then the corresponding sequence of binary matrices converges to the graphon that is identically equal to $p$ on $[0,1]^2$. If $p>0$, Theorem \ref{mainthm} tells us that this sequence of revelation patterns admits stable recovery of low rank matrices. On the other hand, consider our running counterexample, where only the top half of the rows are revealed. The corresponding sequence of binary matrices converges to the graphon that is $1$ in $[0,1/2]\times [0,1]$ and $0$ in $(1/2,1]\times [0,1]$. Therefore this sequence does not admit stable recovery of low rank matrices, as we observed before. 

At this point the reader may be slightly puzzled by the fact that Theorem~\ref{mainthm} implies that recovery is impossible if the set of  revealed entries is sparse (because then the limit graphon is identically zero), whereas there are many existing results about recoverability of low rank matrices from a sparse set of revealed entries. The reason is that we are not assuming randomness and at the same time demanding that the recovery is `stable'. Suppose that most entries are the same for two matrices, but the entries that differ are the only ones that are revealed. Then there is no way to tell that the matrices are mostly the same. Thus, stable recovery is impossible from a small set of revealed entries if there is no assumption of randomness.

Theorem~\ref{mainthm} succeeds in giving an intrinsic characterization of recoverability in terms of the locations of revealed entries. However, it does not tell us how to actually recover a matrix from a set of revealed entries when recovery is possible. Fortunately, it turns out that this is doable by a small modification of an algorithm that is already used in practice, namely, the Cand\`es--Recht algorithm for matrix completion by nuclear norm minimization~\cite{candesrecht09, candestao10, candesplan10}.  The Cand\`es--Recht estimator of  a partially revealed matrix $A$ is the matrix with minimum nuclear norm among all matrices that agree with $A$ at the revealed entries. The modified estimator is the following.
\begin{defn}\label{estdef}
Let $A$ be a matrix whose entries are partially revealed. Suppose that $\|A\|_\infty\le L$ for some known constant $L$. We define the \emph{modified Cand\`es--Recht estimator} of $A$ as the matrix that minimizes nuclear norm among all $B$ that agree with $A$ at the revealed entries and satisfy $\|B\|_\infty\le L$. 
\end{defn}
The assumption of a known upper bound on the $\ell^\infty$ norm is not unrealistic. Usually such upper bounds are known, for example in recommender systems. The modified estimator is the solution of a convex optimization problem, just like the original estimator, and should therefore be computable on a computer if the dimensions are not too large. The following theorem shows that this algorithm is able to approximately recover the full matrix whenever the pattern of revealed entries allows stable recovery.
\begin{thm}\label{convthm}
Let $\{P_k\}_{k\ge 1}$ be a sequence of binary matrices with dimensions tending to infinity that admits stable recovery of low rank matrices. Let $\{A_k\}_{k\ge1}$ be a sequence of matrices such that for each $k$, $A_k$ has the same dimensions as $P_k$. Suppose that $\rank(A_k)$ and $\|A_k\|_\infty$ are uniformly bounded over $k$. Let $\hat{A}_k$ be the modified Cand\`es--Recht estimate of $A_k$ (as defined in Definition \ref{estdef}) when the locations of the revealed entries are defined by $P_k$. Then  $\lim_{k\to\infty} \|\hat{A}_k-A_k\|_{\bar{F}} =0$. 
\end{thm}
The modified Cand\`es--Recht estimator, just like the original estimator, will run into computational cost issues for very large matrices. It would be interesting to figure out if there is a faster algorithm (for example, by some kind of singular value thresholding~\cite{kmo10, chatterjee15, gd14}) that also has the above  `universal recovery' feature.

Another interesting and important problem is to develop an analog of the above theory when the set of revealed entries is sparse. As noted before, the problem is unsolvable in this setting if we demand that the recovery be stable. However, dropping the stability requirement may render it possible to recover the full matrix from a sparse set of revealed entries even in the absence of randomness. In particular, Theorem \ref{convthm} may have an extension to the sparse setting under appropriate assumptions. The methods of this paper would need to be significantly extended to make this possible.  

This concludes the statements of results. The proofs are organized as follows. The proof of Theorem \ref{mainthm} is divided between Sections~\ref{quantproof} and \ref{mainproof}. Theorem \ref{convthm} is proved in Section \ref{convproof}, and Theorem \ref{compactthm} is proved in Section~\ref{compactproof}. 



\section{Some related literature}\label{litsec}
A small number of papers in the literature investigate the problem of matrix completion when the entries are not missing uniformly at random. As mentioned earlier, this list is minuscule in comparison to the vast body of literature on matrix completion under the assumption of missing uniformly at random. The following is a  non-exhaustive list of some of the notable contributions. 

\citet{bj14} developed a general recovery method when the pattern of revealed entries is the adjacency matrix of a bipartite graph with a large spectral gap. A similar question was investigated in a slightly different setting by \citet{hss14}. Inhomogeneous --- but still independent and random --- patterns of missing entries were studied by \citet{cbsw}. 

A very interesting paper of \citet{sc10} applied rigidity theory to understand whether a partially revealed low rank matrix is completable or not. However, this paper did not give an algorithm for completion when completion is possible. A comparison of the criterion from \cite{sc10} with our Theorem \ref{mainthm} is an interesting question that merits further investigation. 

An attempt at giving a criterion for completability using algebraic geometry was made by \citet{ktt15}. This paper has a criterion for recoverability of specific entries of the matrix. But again, a recovery algorithm was not given. Algebraic criteria have also been investigated in other recent papers, such as the one by \citet{pbn16}. 

A different approach was taken by \citet{ls13}, who proposed a new way of interpreting the quality of the output of a given matrix completion algorithm  under arbitrary patterns of missing entries. 

The idea of using the EM algorithm for matrix completion under data-dependent missing patterns was recently studied by \citet{sbj18}. 

The main advantage of our results over most of the papers mentioned above is that we give a condition that is both necessary and sufficient for completability, and also demonstrate that a small modification of a popular algorithm can do the job when it is doable. The main disadvantage, on the other hand, is that our results are of an asymptotic nature. Developing non-asymptotic versions is an important goal. This is further discussed in the next section. 

\section{Open problems}\label{opensec}
The results of this paper leave a lot of questions unanswered. The following is a partial list.
\begin{enumerate}
\item The definition of `stable recovery' entails that the revealed entries are only approximately equal to the corresponding entries of the unknown matrix. What if we drop this condition and assume that the revealed entries are exactly equal to the true entries? How should the theory be modified?
\item Developing non-asymptotic versions of Theorems \ref{mainthm} and \ref{convthm} is extremely desirable. Note that it is not quite clear what should be the proper non-asymptotic statements that one can aspire to prove. A precise   formulation of the non-asymptotic problem is itself an open question. The non-asymptotic formulation is needed for dealing with sparse recovery problems, for the following reason. The theorems of this paper have meaningful implications when the fraction of revealed entries remains fixed as the size of the matrix goes to infinity. To properly understand the level of sparsity allowable for a matrix of a given size, one needs a non-asymptotic result. 
\item It is not clear if the Cand\`es--Recht algorithm indeed needs to be modified, or if the original version is good enough for Theorem \ref{convthm}. We believe that the modification is necessary, but we do not have a counterexample to show that the original algorithm will not work. 
\item As mentioned before, the Cand\`es--Recht algorithm is rather slow for very large matrices. Is there a faster algorithm that can take its place in Theorem \ref{convthm}?
\end{enumerate}

\section{Towards the proof of Theorem \ref{mainthm}}\label{quantproof}
The goal of this section is to prove a quantitative result that underlies the proof of Theorem \ref{mainthm}. We need to prove a number of lemmas before arriving at this theorem.

\begin{lmm}\label{xylmm1}
Let $X$ be an $m\times n$ matrix with $\|X\|_\infty\le 1$ and singular value decomposition 
\[
X = \sum_{i=1}^k \sigma_i u_i v_i^T.
\]
Then for each $i$,
\[
 \sigma_i\le \sqrt{mn}, \ \  \|u_i\|_\infty \le \frac{\sqrt{n}}{\sigma_i} \ \  \text{and} \  \  \|v_i\|_\infty \le \frac{\sqrt{m}}{\sigma_i}.
\]
\end{lmm}
\begin{proof}
Let $u_{ij}$ denote the $j^{\textup{th}}$ component of $u_i$. Since $Xv_i = \sigma_i u_i$ and $\|X\|_\infty\le 1$, we get 
\begin{align*}
\sigma_i |u_{ij}|&\le \sum_{l=1}^n |x_{jl} v_{il}|\le \sum_{l=1}^n |v_{il}| \le \biggl(n\sum_{l=1}^n v_{il}^2\biggr)^{1/2}=  \sqrt{n}. 
\end{align*}
Dividing throughout by $\sigma_i$ and maximizing over $j$, we get the required bound for $\|u_i\|_\infty$. The bound for $\|v_i\|_\infty$ is obtained similarly. For the bound on $\sigma_i$, notice that since $\sum_j u_{ij}^2=1$ there is some $j$ such that $|u_{ij}|\ge m^{-1/2}$, and use this information in the above display. 
\end{proof}
Recall that a matrix is called a {\it block matrix}  if its entries are constant in rectangular blocks --- in other words, if the matrix can be expressed as an array of constant matrices. We will say that two matrices $A$ and $B$ have a simultaneous block structure if they are both block matrices and the rows and columns defining the blocks are the same. Note that block structures may not be uniquely defined, but that will not be a problem for us. 
\begin{lmm}\label{mainlmm1}
Let $X$ and $Y$ be  $m\times n$ matrices with $\|X\|_\infty\le 1$ and $\|Y\|_\infty\le 1$. Let $q\ge 1$ be a number such that $\|X\|_*$ and $\|Y\|_*$ are bounded by $q\sqrt{mn}$. Take any $\ve\in (0,1)$. Then, there exist $m\times n$ matrices $A$ and $B$ with a simultaneous block structure with at most $(20000 q^6 \ve^{-10})^{5q^2\ve^{-2}}$ blocks, and permutations $\pi\in S_m$ and $\tau\in S_n$, such that $\|X^{\pi,\tau}-A\|_{\bar{F}}\le \ve$ and $\|Y^{\pi, \tau}-B\|_{\bar{F}}\le\ve$.  Moreover, it can be ensured that $\|A\|_\infty\le 1$ and $\|B\|_\infty\le1$.
\end{lmm}
\begin{proof}
Fix some $\ve >0$. Let $\delta$, $\gamma$ and $\eta$ be three other positive numbers, to be chosen later. Let 
\begin{align*}
X = \sum_{i=1}^k \sigma_i u_i v_i^T, \ \ \ Y = \sum_{i=1}^l \lambda_i w_i z_i^T
\end{align*}
be the singular value decompositions of $X$ of $Y$, with $\sigma_1\ge \cdots \ge \sigma_k>0$ and $\lambda_1\ge \cdots \ge \lambda_l>0$. Choose two numbers $k_1$ and $l_1$ such that $\sigma_{k_1}> \delta \ge \sigma_{k_1+1}$ and $\lambda_{l_1}> \delta \ge \lambda_{l_1+1}$. If $\sigma_i\le \delta$ for all  $i$, let $k_1=0$, and if $\sigma_i> \delta$ for all $i$, let $k_1=k$. Similarly, if $\lambda_i\le \delta$ for all  $i$, let $l_1=0$, and if $\lambda_i> \delta$ for all $i$, let $l_1=l$. Let 
\[
X_1 := \sum_{i=1}^{k_1} \sigma_i u_i v_i^T, \ \ \ Y_1 := \sum_{i=1}^{l_1} \lambda_i w_i z_i^T.
\]
Then by the definition of $k_1$, 
\begin{align}
\|X-X_1\|_{\bar{F}}^2 &= \frac{1}{mn} \sum_{i=k_1+1}^k \sigma_i^2 \notag \\
&\le \frac{\delta}{mn}\sum_{i=1}^k\sigma_i = \frac{\|X\|_*\delta}{mn} \le \frac{q\delta}{\sqrt{mn}}.\label{xbd1}
\end{align}
Similarly, the same bound holds for $\|Y-Y_1\|_{\bar{F}}^2$.

For $1\le i\le k$ and $1\le a\le m$, let $u_{ia}$ denote the $a^{\textup{th}}$ component of $u_i$. Define $\tilde{u}_{ia}$ to be the  integer multiple of $\gamma$ that is closest to $u_{ia}$ under the constraint that $|\tilde{u}_{ia}|\le |u_{ia}|$. Then note that $|u_{ia}-\tilde{u}_{ia}|\le \gamma$. Let $\tilde{u}_i$ be the vector whose $a^{\textup{th}}$ component is $\tilde{u}_{ia}$. Let $\tilde{w}_i$ be defined similarly. Define $\tilde{v}_i$ and $\tilde{z}_i$ the same way, but using $\eta$ instead of $\gamma$. Let
\[
\tilde{X}_1 := \sum_{i=1}^{k_1} \sigma_i \tilde{u}_i \tilde{v}_i^T, \ \ \ \tilde{Y}_1 := \sum_{i=1}^{l_1} \lambda_i \tilde{w}_i \tilde{z}_i^T.
\]
Now take any $1\le i\le k_1$. By Lemma \ref{xylmm1} and the choice of $k_1$, we have
\[
\|u_i\|_\infty < \frac{\sqrt{n}}{\delta}. 
\]
Therefore for any $1\le a\le m$, the set of possible values of $\tilde{u}_{ia}$ has size at most 
\[
\frac{2\sqrt{n}}{\delta\gamma} + 1\le \frac{4\sqrt{n}}{\delta\gamma},
\]
where the inequality was obtained under the assumption that 
\begin{align}\label{deltagamma}
\frac{2\sqrt{n}}{\delta\gamma}\ge 1.
\end{align}
We will later choose $\delta$ and $\gamma$ such that this assumption is valid. We can give similar bounds on the sizes of the sets of possible values of the components of $\tilde{v}_i$, $\tilde{w}_i$ and~$\tilde{z}_i$.

Declare that two rows $a$ and $a'$ are `equivalent' if $\tilde{u}_{ia}=\tilde{u}_{ia'}$ and $\tilde{w}_{i'a}=\tilde{w}_{i'a'}$ for all $1\le i\le k_1$ and $1\le i'\le l_1$. Similarly declare that two columns $b$ and $b'$ are equivalent if $\tilde{v}_{ib}=\tilde{v}_{ib'}$ and $\tilde{z}_{i'b}=\tilde{z}_{i'b'}$ for all $1\le i\le k_1$ and $1\le i'\le l_1$. Clearly, these define equivalence relations. By the previous paragraph, there are at most $(4\sqrt{n}/(\delta\gamma))^{k_1+l_1}$ equivalence classes of rows, and at most $(4\sqrt{m}/(\delta\eta))^{k_1+l_1}$ equivalence classes of columns. 

Let $\pi$ be a permutation of the rows that `clumps together' equivalent rows, and let $\tau$ be a permutation of the columns that clumps together equivalent columns. Then it is clear that $\tilde{X}_1^{\pi, \tau}$ and $\tilde{Y}_1^{\pi,\tau}$ are block matrices. By the previous paragraph, the number of blocks is at most $(16\sqrt{mn}/(\delta^2\gamma\eta))^{k_1+l_1}$. 

Now note that
\[
k_1\sigma_{k_1}\le \sum_{i=1}^k \sigma_i =\|X\|_* \le q\sqrt{mn}.
\]
But $\sigma_{k_1}> \delta$. Thus, 
\[
k_1 \le \frac{q\sqrt{mn}}{\delta}.
\]
Similarly, $l_1$ is also bounded by the same quantity. Thus, the number of blocks is at most 
\[
b := \biggl(\frac{16\sqrt{mn}}{\delta^2\gamma\eta}\biggr)^{2q\sqrt{mn}/\delta}.
\]
Now notice that by Lemma \ref{xylmm1} and the definition of $\tilde{X}_1$, 
\begin{align*}
\|X_1-\tilde{X}_1\|_{\bar{F}} &\le \|X_1-\tilde{X}_1\|_\infty \\
&\le \sum_{i=1}^{k_1} \sigma_i (\|u_i-\tilde{u}_i\|_\infty\|v_i\|_\infty + \|\tilde{u}_i\|_\infty \|v_i-\tilde{v}_i\|_\infty) \\
&\le q \sqrt{mn} \biggl(\frac{\sqrt{m}\gamma}{\delta} + \frac{\sqrt{n}\eta}{\delta}\biggr).
\end{align*}
By a similar argument, the same bound holds for  $\|Y_1-\tilde{Y}_1\|_{\bar{F}}$. Combining with \eqref{xbd1}, we see that  if $A:= \tilde{X}_1^{\pi, \tau}$ and $B := \tilde{Y}_1^{\pi, \tau}$, then $A$ and $B$ have a simultaneous block structure with at most  $b$ blocks, and $\|X^{\pi,\tau}-A\|_{\bar{F}}$ and $\|Y^{\pi,\tau}-B\|_{\bar{F}}$ are bounded by 
\begin{align*}
\frac{\sqrt{q\delta}}{(mn)^{1/4}} + q \sqrt{mn} \biggl(\frac{\sqrt{m}\gamma}{\delta} + \frac{\sqrt{n}\eta}{\delta}\biggr).
\end{align*}
Now take any $\alpha, \beta>0$ and define 
\[
\delta := \alpha\sqrt{mn}, \ \ \gamma := \frac{\beta}{\sqrt{m}}, \ \  \eta :=  \frac{\beta}{\sqrt{n}}.
\]
Plugging these values into the previous display gives 
\[
\sqrt{q\alpha} + \frac{2q\beta}{\alpha}. 
\]
For a given $\beta$, the above quantity is minimized by taking $\alpha = 2^{4/3} q^{1/3}\beta^{2/3}$, and the minimum value is $(2^{2/3}+2^{-1/3}) q^{2/3}\beta^{1/3}$. Choose $\beta$ to make this equal to $\ve$, which ensures that $\|X^{\pi,\tau}-A\|_{\bar{F}}$ and $\|Y^{\pi,\tau}-B\|_{\bar{F}}$ are bounded by $\ve$. With these choices of $\alpha$ and $\beta$, an easy calculation gives
\[
b = (16\alpha^{-2}\beta^{-2})^{2q/\alpha} \le (20000 q^6 \ve^{-10})^{5q^2\ve^{-2}}. 
\]
Also, it is easy to check (using $q\ge1$ and $\ve \in (0,1)$) that with these choices of $\alpha$ and $\beta$, the inequality \eqref{deltagamma} holds. Thus, the proof is complete except that we have not ensured that $\|A\|_\infty\le1$ and $\|B\|_\infty\le1$ in our construction. To force this, just take any element of either matrix; if it is bigger than $1$, replace it by $1$; if it is less than $-1$, replace it by $-1$. This retains the block structures of the matrices, and it cannot increase $|x_{\pi(i)\tau(j)}-a_{ij}|$ or $|y_{\pi(i)\tau(j)}-b_{ij}|$ for any $i,j$ since $|x_{\pi(i)\tau(j)}|\le 1$ and $|y_{\pi(i)\tau(j)}|\le 1$.
\end{proof}

\begin{lmm}\label{mainlmm2}
Let $A$ and $B$ be $m\times n$ matrices with a simultaneous block structure. Let $b$ be the number of blocks. Let $P$ and $Q$ be $m\times n$ matrices such that $P$ is binary and the entries of $Q$ are all in $[0,1]$. Then 
\begin{align*}
\|(A-B)\circ Q\|_{\bar{F}} &\le \|(A-B)\circ P\|_{\bar{F}} + \sqrt{b\|P-Q\|_\Box }\, \|A-B\|_\infty. 
\end{align*}
\end{lmm}
\begin{proof}
Let each block be represented by the set of pairs of indices $(i,j)$ that belong to the block. Let $\md$ be the set of all blocks. Take any block $D\in \md$. By the definition the cut norm,
\begin{align}\label{cutconseq}
\frac{1}{mn}\biggl|\sum_{(i,j)\in D}(p_{ij}-q_{ij})\biggr|\le \|P-Q\|_\Box.
\end{align}
Recall that $a_{ij}$ is the same for all $(i,j)$ in a block, and the same holds for $b_{ij}$. Let $a(D)$ and $b(D)$ denote the values of $a_{ij}$ and $b_{ij}$ in  a block $D$. Since $q_{ij}\in [0,1]$ for all $i,j$, 
\begin{align*}
\|(A-B)\circ Q\|_{\bar{F}}^2 &= \frac{1}{mn}\sum_{i,j} (a_{ij}-b_{ij})^2 q_{ij}^2 \\
&\le  \frac{1}{mn}\sum_{i,j} (a_{ij}-b_{ij})^2 q_{ij}\\
&= \sum_{D\in \md} (a(D)-b(D))^2 \biggl(\frac{1}{mn}\sum_{(i,j)\in D} q_{ij}\biggr).
\end{align*}
Therefore by \eqref{cutconseq},
\begin{align*}
\|(A-B)\circ Q\|_{\bar{F}}^2 &\le \sum_{D\in \md} (a(D)-b(D))^2 \biggl(\frac{1}{mn}\sum_{(i,j)\in D} p_{ij}+ \|P-Q\|_\Box\biggr).
\end{align*}
Since $P$ is binary, $p_{ij}=p_{ij}^2$ for all $i,j$.  Thus, we get
\begin{align*}
\|(A-B)\circ Q\|_{\bar{F}}^2 &\le \frac{1}{mn}\sum_{D\in \md}\sum_{(i,j)\in D} (a_{ij}-b_{ij})^2 p_{ij}^2\\
&\qquad+ \|P-Q\|_\Box\|A-B\|_\infty^2|\md|\\
&= \|(A-B)\circ P\|_{\bar{F}}^2 + \|P-Q\|_\Box \|A-B\|_\infty^2 b. 
\end{align*}
The proof is not completed by applying the inequality $\sqrt{x+y}\le \sqrt{x}+\sqrt{y}$ to the right side. 
\end{proof}
We are now ready to prove the main result of this section. The result roughly says the following. Let $X$ and $Y$ be  matrices with relatively small nuclear norms (of the same order as that for low rank matrices). Let $P$ be a binary matrix and $Q$ be a matrix with entries in $[0,1]$, such that $P$ is close to $Q$ in the cut norm. Then, the closeness of $X\circ P$ to $Y \circ P$ in average Frobenius norm implies the closeness of $X\circ Q$ to $Y\circ Q$ in  average Frobenius norm. 
\begin{thm}\label{quantthm}
Let $X$ and $Y$ be  $m\times n$ matrices with $\ell^\infty$ norms bounded by~ $1$. Let $q$ be a number such that the nuclear norms of $X$ and $Y$  are bounded by $q\sqrt{mn}$. Let $P$ and $Q$ be $m\times n$ matrices such that $P$ is binary and the entries of $Q$ are all in $[0,1]$. Then
\[
\|(X-Y)\circ Q\|_{\bar{F}}\le \|(X-Y)\circ P\|_{\bar{F}} + C(q)\sqrt{\frac{\log (-\log \|P-Q\|_\Box)}{-\log \|P-Q\|_\Box}},
\]
where $C(q)$ depends only on $q$.
\end{thm}

\begin{proof}
Without loss of generality, assume that $q\ge1$.   
Take any $\ve >0$. Let $A$, $B$, $\pi$ and $\tau$ be as in Lemma~\ref{mainlmm1}. Let 
\[
b := (20000 q^6 \ve^{-10})^{5q^2\ve^{-2}}
\] 
be the upper bound on the number of blocks given by Lemma \ref{mainlmm1}. 
Note that 
\begin{align*}
&\|(X-Y)\circ Q\|_{\bar{F}} = \| (X^{\pi,\tau}-Y^{\pi,\tau})\circ Q^{\pi, \tau}\|_{\bar{F}}\\
&\le \| (X^{\pi,\tau}-A)\circ Q^{\pi, \tau}\|_{\bar{F}} + \| (A-B)\circ Q^{\pi, \tau}\|_{\bar{F}} + \| (B-Y^{\pi,\tau})\circ Q^{\pi, \tau}\|_{\bar{F}}\\
&\le \|X^{\pi,\tau}-A\|_{\bar{F}} + \| (A-B)\circ Q^{\pi, \tau}\|_{\bar{F}}+ \| B-Y^{\pi,\tau}\|_{\bar{F}}\\
&\le 2\ve + \| (A-B)\circ Q^{\pi, \tau}\|_{\bar{F}}.
\end{align*}
By Lemma \ref{mainlmm2}, 
\begin{align*}
&\| (A-B)\circ Q^{\pi, \tau}\|_{\bar{F}} \\
&\le \| (A-B)\circ P^{\pi, \tau}\| _{\bar{F}}+ \sqrt{b\|P^{\pi,\tau}-Q^{\pi, \tau}\|_\Box} \, \|A-B\|_\infty\\
&\le \| (A-B)\circ P^{\pi, \tau}\| _{\bar{F}}+ 2\sqrt{b\|P-Q\|_\Box}.
\end{align*}
But 
\begin{align*}
&\| (A-B)\circ P^{\pi, \tau}\|_{\bar{F}}\\
&\le \| (A-X^{\pi,\tau})\circ P^{\pi, \tau}\|_{\bar{F}} + \| (X^{\pi,\tau}-Y^{\pi,\tau})\circ P^{\pi, \tau}\|_{\bar{F}} \\
&\qquad + \| (Y^{\pi,\tau}-B)\circ P^{\pi, \tau}\|_{\bar{F}}\\
&\le \|A-X^{\pi,\tau}\|_{\bar{F}} + \| (X-Y)\circ P\|_{\bar{F}}+ \|Y^{\pi,\tau}-B\|_{\bar{F}}\\
&\le 2\ve + \| (X-Y)\circ P\|_{\bar{F}}.
\end{align*}
Adding up, we get
\begin{align*}
&\|(X-Y)\circ Q\|_{\bar{F}} \le \| (X-Y)\circ P\|_{\bar{F}} + 4\ve + 2\sqrt{b\|P-Q\|_\Box}\\
&\le \| (X-Y)\circ P\|_{\bar{F}} + 4\ve + (C q^6 \ve^{-10})^{3q^2\ve^{-2}}\sqrt{\|P-Q\|_\Box},
\end{align*}
where $C$ is a universal constant. This bound holds for $\ve\in (0,1)$, but it also holds for $\ve \ge 1$ due to the presence of the $4\ve$ term. 
The required bound is now obtained by choosing 
\[
\ve = C(q)\sqrt{\frac{\log (-\log \|P-Q\|_\Box)}{-\log \|P-Q\|_\Box}}
\]
for some sufficiently large constant  $C(q)$ depending only on $q$.
\end{proof}

\section{Proof of Theorem \ref{mainthm}}\label{mainproof}
Let $\{P_k\}_{k\ge 1}$ be a sequence of binary matrices converging to a graphon $W$. Suppose that $W$ is nonzero everywhere. Let $m_k$ and $n_k$ be the number of rows and the number columns in $P_k$. Suppose that $m_k$ and $n_k$ tend to infinity as $k\to\infty$. We will first prove the following generalization of the `if' part of Theorem \ref{mainthm}.
\begin{thm}\label{mainthm2}
Let $P_k$ and $W$ be as above. Take any two sequences of $m_k\times n_k$ matrices $\{A_k\}_{k\ge 1}$ and $\{B_k\}_{k\ge 1}$. Suppose that there are numbers $q$ and $L$ such that $\|A_k\|_*$ and $\|B_k\|_*$ are bounded by $q\sqrt{m_kn_k}$ and $\|A_k\|_\infty$ and $\|B_k\|_\infty$ are bounded by $L$ for each $k$. Then for any $\ve>0$ there is some $\delta>0$, depending only on $\ve$, $q$ and $L$, such that if  $\limsup_{k\to\infty} \|(A_k - B_k)\circ P_k\|_{\bar{F}} \le \delta$, 
then $\limsup_{k\to\infty} \|A_k-B_k\|_{\bar{F}} \le \ve$.
\end{thm}
\begin{proof}
For simplicity of notation, let us denote the matrix $W_{m_k,n_k}$ by $W^{(k)}$. The convergence of $P_k$ to $W$ means that for each $k$, there are permutations $\pi_k\in S_{m_k}$ and $\tau_k\in S_{n_k}$ such that 
\begin{align*}
\lim_{k\to\infty} \|P_k^{\pi_k, \tau_k} -W^{(k)}\|_\Box = 0.
\end{align*}
Rearranging the rows and columns of $P_k$, $A_k$ and $B_k$, we may assume without loss of generality that $\pi_k$ and $\tau_k$ are the identity permutations, so that 
\begin{align}\label{mainlim1}
\lim_{k\to\infty} \|P_k -W^{(k)}\|_\Box = 0.
\end{align}
Let $\delta>0$ be a number such that
\begin{align}\label{mainlim2}
\limsup_{k\to\infty} \|(A_k - B_k)\circ P_k\|_{\bar{F}} \le \delta.
\end{align}
Without loss of generality, $L=1$. Then by \eqref{mainlim1}, \eqref{mainlim2} and Theorem \ref{quantthm}, 
\begin{align}\label{mainlim3}
\limsup_{k\to\infty} \|(A_k-B_k)\circ W^{(k)}\|_{\bar{F}} &\le \delta. 
\end{align}
Now take any $\eta\in (0,1)$. Define two functions $f,g:[0,1]\to[0,1]$ as
\[
f(x) := 
\begin{cases}
\eta &\text{ if } x\le \eta,\\
x &\text{ if } x>\eta,
\end{cases}
\]
and $g(x) := (f(x)-x)/\eta$.  Let $U^{(k)}$ be the matrix whose $(i,j)^{\textup{th}}$ element is $f(w^{(k)}_{ij})$ and let  $V^{(k)}$ be the matrix whose $(i,j)^{\textup{th}}$ element is $g(w^{(k)}_{ij})$. Since $f(x)\ge \eta$ for all $x$, 
\begin{align*}
\|A_k-B_k\|_{\bar{F}} &\le \frac{1}{\eta}\|(A_k-B_k)\circ U^{(k)}\|_{\bar{F}}\\
&\le \frac{1}{\eta}\|(A_k-B_k)\circ (U^{(k)}-W^{(k)})\|_{\bar{F}} + \frac{1}{\eta}\|(A_k-B_k)\circ W^{(k)}\|_{\bar{F}}\\
&= \|(A_k-B_k)\circ V^{(k)}\|_{\bar{F}} + \frac{1}{\eta}\|(A_k-B_k)\circ W^{(k)}\|_{\bar{F}}.
\end{align*}
Therefore by \eqref{mainlim3},
\begin{align}\label{mainlim4}
\limsup_{k\to\infty} \|A_k-B_k\|_{\bar{F}} &\le \limsup_{k\to\infty} \|(A_k-B_k)\circ V^{(k)}\|_{\bar{F}} + \frac{\delta}{\eta}. 
\end{align}
Since $L=1$, 
\begin{align*}
\|(A_k-B_k)\circ V^{(k)}\|_{\bar{F}}^2 &\le \frac{4}{m_kn_k}\sum_{i,j} g(w_{ij}^{(k)})^2\\
&= 4\iint g( W^{(k)}(x,y))^2 dx dy, 
\end{align*}
where $W^{(k)}$ now denotes the function which equals $w^{(k)}_{ij}$ for all $(x,y)$ in the rectangle $[\frac{i-1}{m_k}, \frac{i}{m_k}]\times [\frac{j-1}{n_k}, \frac{j}{n_k}]$. In other words, $W^{(k)}$ is obtained by averaging $W$ within each such rectangle. Since $m_k$ and $n_k$ tend to $\infty$ and $W$ is measurable, it follows by a standard result from analysis (see, for example, \cite[Proposition 9.8]{lovaszbook}) that $W^{(k)}(x,y) \to W(x,y)$ as $k\to\infty$ for almost every $(x,y)$. Since $g$ is a bounded continuous function, this shows that 
\begin{align*}
\limsup_{k\to\infty} \|(A_k-B_k)\circ V^{(k)}\|_{\bar{F}}^2 &\le 4\iint g( W(x,y))^2 dx dy, 
\end{align*}
On the other hand, $g(x)\le 1_{\{x\le \delta\}}$ for all $x$. Thus, 
\begin{align*}
\limsup_{k\to\infty} \|(A_k-B_k)\circ V^{(k)}\|_{\bar{F}}^2 &\le 4\phi(\eta), 
\end{align*}
where $\phi(\eta)$ is the Lebesgue measure of the set of all $(x,y)$ where $W(x,y)\le \eta$. Combining with \eqref{mainlim4}, we get
\begin{align*}
\limsup_{k\to\infty} \|A_k-B_k\|_{\bar{F}} &\le 2\sqrt{\phi(\eta)} + \frac{\delta}{\eta}. 
\end{align*}
Note that this holds for any $\eta\in (0,1)$. Since $W$ is nonzero almost everywhere, $\phi(\eta)\to 0$ as $\eta \to 0$. Thus, given $\ve >0$, we can first choose $\eta$ so small that $2\sqrt{\phi(\eta)} \le \ve/2$, and then choose $\delta$ so small that $\delta/\eta \le \ve /2$. If the sequences $A_k$ and $B_k$ satisfy \eqref{mainlim2} with this $\delta$, then the above display allows us to conclude that $\limsup_{k\to\infty} \|A_k-B_k\|_{\bar{F}} \le \ve$. 
\end{proof}
We are now ready to prove Theorem \ref{mainthm}.
\begin{proof}[Proof of Theorem \ref{mainthm}]
The `if' part of Theorem \ref{mainthm} follows immediately from Theorem \ref{mainthm2} and the observation, by Lemma \ref{xylmm1}, that 
\begin{align}\label{nuclearbd}
\|X\|_*\le \rank(X) \|X\|_\infty\sqrt{mn}
\end{align}
for any $m\times n$ matrix $X$.  

For the `only if' part, suppose that $W$ is zero on a set of positive Lebesgue measure. Denote this set by $S$ and let $\lambda(S)$ denote its Lebesgue measure. Take any $\ve >0$. By a standard measure-theoretic argument, there exists $T\subseteq [0,1]^2$ such that $T$ is a union of dyadic squares of equal size and $\lambda(S\Delta T) <\ve$. Let $\md$ be the set of all dyadic squares of this size in $[0,1]^2$. 

For each $k$, let $A_k$ be the  zero matrix of order $m_k\times n_k$. Let $B_k$ be the $m_k\times n_k$ matrix whose $(i,j)^{\textup{th}}$ entry is $1$ if $(\frac{i}{m_k},\frac{j}{n_k})\in T$ and $0$ otherwise. Since $T$ is a union of elements of $\md$, it is not difficult to see that $B_k$ is a block matrix with at most $|\md|$ blocks. In particular, its rank is bounded above  by $|\md|$. 

Now note that $\|A_k-B_k\|_{\bar{F}}^2$ equals the fraction of indices $(i,j)$ such that $(\frac{i}{m_k},\frac{j}{n_k})\in T$. Therefore as $k\to\infty$, $\|A_k-B_k\|_{\bar{F}}^2$ tends to $\lambda(T)$. If $\ve$ is  small enough, this ensures that
\begin{align}\label{counter1}
\lim_{k\to\infty} \|A_k-B_k\|_{\bar{F}}\ge \sqrt{\frac{\lambda(S)}{2}}>0. 
\end{align}
On the other hand, $\|(A_k-B_k)\circ P_k\|_{\bar{F}}^2$ equals the fraction of indices $(i,j)$ such that $(\frac{i}{m_k},\frac{j}{n_k})\in T$ and the $(i,j)^{\textup{th}}$ entry of $P_k$ is $1$. Let $D$ be one of the constituent dyadic cubes of $T$. Let $f_k(D)$ be the fraction of $(i,j)$ such that $(\frac{i}{m_k},\frac{j}{n_k})\in D$ and the $(i,j)^{\textup{th}}$ entry of $P_k$ is $1$. From the definition of cut norm, it follows  that 
\[
\lim_{k\to\infty} f_k(D) = \iint_D W(x,y) dxdy. 
\]
Summing over all $D$ as above, we get
\[
\lim_{k\to\infty} \|(A_k-B_k)\circ P_k\|_{\bar{F}}^2 =\iint_T W(x,y)dx dy.
\]
Since $\lambda(S\Delta T)\le \ve$, $W$ takes values in $[0,1]$,  and $W=0$ on $S$, this shows that 
\begin{align}\label{counter2}
\lim_{k\to\infty} \|(A_k-B_k)\circ P_k\|_{\bar{F}}^2 \le\ve + \iint_S W(x,y)dxdy  = \ve. 
\end{align}
Since $\ve$ is arbitrary, the combination of \eqref{counter1} and \eqref{counter2} shows that the sequence $P_k$ does not admit stable recovery of low rank matrices. This completes the proof of Theorem \ref{mainthm}. 
\end{proof}

\section{Proof of Theorem \ref{convthm}}\label{convproof}
To prove that $\|A_k-\hat{A}_k\|_{\bar{F}}\to 0$, we will show that for any subsequence, there is a further subsequence through which this convergence takes place. By Theorem \ref{compactthm}, we know that any subsequence has a further subsequence along which $P_k$ converges to a limit graphon. Moreover, it is easy to see that if a sequence of binary matrices admits stable recovery of low rank matrices, then any subsequence also does so. Therefore by Theorem \ref{mainthm}, we may assume without loss of generality that $P_k\to W$ for some $W$ that is nonzero almost everywhere. 

Also without loss of generality, suppose that $\|A_k\|_\infty\le1$ for all $k$. Let $L$ be a uniform upper bound on $\rank(A_k)$. Then by \eqref{nuclearbd},
\[
\|A_k\|_* \le L\sqrt{m_kn_k},
\]
where $m_k$ and $n_k$ are the number of rows and number of columns in $A_k$. Consequently, $\|\hat{A}_k\|_*$ is also bounded by $L\sqrt{m_kn_k}$. Moreover, by construction, $\|\hat{A}_k\|_\infty\le 1$ and $\|(A_k-\hat{A}_k)\circ P_k\|_{\bar{F}}=0$ for all $k$. Therefore by Theorem~\ref{mainthm2}, we can now conclude that $\|\hat{A}_k-A_k\|_{\bar{F}}\to 0$ as $k\to\infty$. 

\section{Proof of Theorem \ref{compactthm}}\label{compactproof}
Let $m$ and $n$ be two positive integers. Let $\cp$ be a partition of $\{1,\ldots,m\}$ and let $\cq$ be a partition of $\{1,\ldots, n\}$. The pair $(\cp,\cq)$ defines a block structure for $m\times n$ matrices in the natural way: Two pairs of indices $(i,j)$ and $(i',j')$ belong to the same block if and only if $i$ and $i'$ belong to the same member of $\cp$ and $j$ and $j'$ belong to the same member of $\cq$. 

If $A$ is an $m\times n$ matrix, let $A^{\cp,\cq}$ be the `block averaged' version of $A$, obtained by replacing the entries in each block (in the block structure defined by $(\cp,\cq)$) by the average value in that block. It is easy to see from the definition of the cut norm that 
\begin{equation}\label{boxcontract}
\|A^{\cp,\cq}\|_\Box \le \|A\|_\Box.
\end{equation}
We need the following lemma.
\begin{lmm}\label{szemlmm}
For any $m\times n$ matrix $A$ with $\|A\|_\infty\le 1$, there is a sequence of partitions $\{\cp_j\}_{j\ge 1}$ of $\{1,\ldots,m\}$ and a sequence of partitions $\{\cq_j\}_{j\ge 1}$ of $\{1,\ldots, n\}$ such that for each $j$,
\begin{enumerate}
\item $\cp_{j+1}$ is a refinement of $\cp_j$ and $\cq_{j+1}$ is a refinement of $\cq_j$, 
\item $|\cp_j|$ and $|\cq_j|$ are bounded by $(2^{j+2}j)^{j^2}$, and
\item $\|A - A^{\cp_j,\cq_j}\|_\Box\le 2j^{-1}+6j^3 2^{-j}$.
\end{enumerate}
\end{lmm}
\begin{proof}
Let 
\[
A = \sum_{i=1}^r \sigma_i u_i v_i^T
\]
be the singular value decomposition of $A$, where $\sigma_1\ge\cdots \ge \sigma_r>0$ are the nonzero singular values. Take any $j\ge 1$. Let $l$ be the largest number such that $\sigma_i> \sqrt{mn}/j$. If there is no such $l$, let $l=0$. Let 
\[
A_1 := \sum_{i=1}^l \sigma_i u_i v_i^T.
\]
For $1\le i\le l$ and $1\le a\le m$, let $u_{ia}$ denote the $a^{\textup{th}}$ component of $u_i$. Let $\tilde{u}_{ia}^{(j)}$ be the largest integer multiple of $2^{-j}m^{-1/2}$ that is $\le u_{ia}$. Let $\tilde{u}_i^{(j)}$ be the vector whose $a^{\textup{th}}$ component is $\tilde{u}_{ia}^{(j)}$. Similarly, for $1\le b\le n$, let  $\tilde{v}_{ib}^{(j)}$ be the largest integer multiple of $2^{-j}n^{-1/2}$ that is $\le v_{ib}^{(j)}$. Define
\[
\tilde{A}_1 := \sum_{i=1}^l \sigma_i \tilde{u}_i \tilde{v}_i^T.
\]
Declare that two rows $a$ and $a'$ are equivalent if $\tilde{u}_{ia}^{(j)}=\tilde{u}_{ia'}^{(j)}$ for all $1\le i\le l$. Similarly declare that two columns $b$ and $b'$ are equivalent if $\tilde{v}_{ib}^{(j)}=\tilde{v}_{ib'}^{(j)}$ for all $1\le i\le l$. Let $\cp_j$ be the set of equivalence classes of rows and $\cq_j$ be the set equivalence classes of columns.

From the above definition, it is clear that if $\tilde{u}_{ia}^{(j+1)} = \tilde{u}_{ia'}^{(j+1)}$, then $\tilde{u}_{ia}^{(j)} = \tilde{u}_{ia'}^{(j)}$. This shows that $\cp_{j+1}$ is a refinement of $\cp_j$. Similarly, $\cq_{j+1}$ is a refinement of $\cq_j$. 

Next, note that by Lemma \ref{xylmm1} and the definition of $l$, 
\[
\|u_i\|_\infty \le \frac{\sqrt{n}}{\sigma_i} \le \frac{j}{\sqrt{m}}
\]
for $1\le i\le l$. Thus, the set of possible values of $\tilde{u}_{ia}$ has size at most 
\[
\frac{2j/\sqrt{m}}{2^{-j}/\sqrt{m}} + 1 = 2^{j+1}j+1\le 2^{j+2}j.
\]
Therefore, $|\cp_j|\le (2^{j+2}j)^l$. Now,
\begin{align*}
l\sigma_l^2 \le \sum_{i=1}^r \sigma_i^2 = \|A\|_F^2 \le mn,
\end{align*}
where the last inequality holds because $\|A\|_\infty\le 1$. Since $\sigma_l\ge \sqrt{mn}/j$, this gives 
\begin{align}\label{lbd}
l\le j^2.
\end{align}
Thus, $|\cp_j|\le (2^{j+2} j)^{j^2}$. Similarly, $|\cq_j|\le (2^{j+2}j)^{j^2}$.

Now recall that the operator norm $\|M\|_{op}$ of a matrix $M$ is the maximum of $\|Mx\|$ over all vectors $x$ with $\|x\|\le 1$. The operator norm of a matrix is equal to its largest singular value. From our definition of the cut norm, it is not difficult to see that  for an $m\times n$ matrix $M$, 
\[
\|M\|_\Box \le \frac{\|M\|_{op}}{\sqrt{mn}}.
\]
Thus, 
\begin{align*}
\|A-A_1\|_\Box \le \frac{\|A-A_1\|_{op}}{\sqrt{mn}}= \frac{\sigma_{l+1}}{\sqrt{mn}}\le \frac{1}{j}.
\end{align*}
Next, by \eqref{lbd}, Lemma \ref{xylmm1} and the (easy) fact that the cut norm is bounded above by the average Frobenius norm,
\begin{align*}
&\|A_1-\tilde{A}_1\|_\Box  \le \|A_1-\tilde{A}_1\|_{\bar{F}} \\
&\le \sum_{i=1}^l \sigma_i (\|u_i-\tilde{u}_i\|_\infty \|v_i\|_\infty + \|\tilde{u}_i\|_\infty \|v_i - \tilde{v}_i\|_\infty)\\
&\le \sum_{i=1}^l \sigma_i (\|u_i-\tilde{u}_i\|_\infty \|v_i\|_\infty + (\|\tilde{u}_i - u_i\|_\infty +\|u_i\|_\infty)\|v_i - \tilde{v}_i\|_\infty)\\
&\le l\sqrt{mn}(2^{-j} m^{-1/2} j n^{-1/2}+ 2^{-j}m^{-1/2} 2^{-j}n^{-1/2} + jm^{-1/2}2^{-j}n^{-1/2})\\
&\le 3j^3 2^{-j}. 
\end{align*}
Combining, we get
\[
\|A-\tilde{A}_1\|_\Box \le j^{-1}+3j^3 2^{-j}. 
\]
Now note that $\tilde{A}_1$ is constant within the blocks defined by the pair $(\cp_j,\cq_j)$. Thus, by \eqref{boxcontract},
\begin{align*}
\|A-A^{\cp_j, \cq_j}\|_\Box &\le\|A-\tilde{A}_1\|_\Box + \|\tilde{A}_1-A^{\cp_j, \cq_j}\|_\Box\\
&\le \|A-\tilde{A}_1\|_\Box + \|\tilde{A}_1^{\cp_j,\cq_j}-A^{\cp_j, \cq_j}\|_\Box\le 2\|A-\tilde{A}_1\|_\Box.
\end{align*}
This completes the proof.
\end{proof}
We are now ready to prove Theorem \ref{compactthm}. In this proof, we will use the following scheme to define a graphon using a matrix. Suppose that $A$ is an $m\times n$ matrix. The graphon defined by $A$, which we will also denote by $A$, is the function $A:[0,1]^2\to[0,1]$ which equals $a_{ij}$ in the rectangle $(\frac{i-1}{m}, \frac{i}{m})\times (\frac{j-1}{m}, \frac{j}{m})$ for each $1\le i\le m$ and $1\le j\le n$. On the boundaries of the rectangles, $A$ can be defined arbitrarily. 

In the proof, we will need to work with cut norms of asymmetric graphons. The cut norm of an asymmetric graphon $W$ is defined as
\[
\|W\|_\Box := \sup_{a,b}\biggl|\iint a(x) b(y) W(x,y) dxdy\biggr|,
\]
where the supremum is taken over all Borel measurable $a,b:[0,1]\to[-1,1]$. If the graphon is defined by a matrix as in the previous paragraph, it is easy to see that the cut norm of the graphon equals the cut norm of the matrix. A property of the cut norm that we will use in the proof is that the cut norm of an asymmetric graphon is bounded above by its $L^1$ norm.  
\begin{proof}[Proof of Theorem \ref{compactthm}]
Let $\{A_k\}_{k\ge 1}$ be a sequence of matrices with dimensions tending to infinity. Let $m_k$ and $n_k$ be the number of rows and number of columns in $A_k$. Lemma~\ref{szemlmm} tells us that for each $k$ and $j$, we can find a partition $\cp_{k,j}$ of $\{1,\ldots,m_k\}$ and a partition $\cq_{k,j}$ of $\{1,\ldots, n_k\}$ such that
\begin{enumerate}
\item $\cp_{k,j+1}$ is a refinement of $\cp_{k,j}$ and $\cq_{k,j+1}$ is a refinement of $\cq_{k,j}$, 
\item $|\cp_{k,j}|$ and $|\cq_{k,j}|$ are bounded by $(2^{j+2}j)^{j^2}$, and 
\item $\|A_k - A_k^{\cp_{k,j},\cq_{k,j}}\|_\Box\le 2j^{-1}+6j^3 2^{-j}$.
\end{enumerate}
For simplicity, let us denote $A_k^{\cp_{k,j},\cq_{k,j}}$ by $A_{k,j}$. By permuting rows and columns of $A_k$, let us assume that the members of $\cp_{k,j}$ and $\cq_{k,j}$ are intervals, so that $A_{k,j}$ is a block matrix. As described in the paragraph preceding this proof, the matrix $A_{k,j}$ defines an asymmetric graphon which is  also denoted by $A_{k,j}$. This graphon is constant in rectangular blocks, where the number of blocks is bounded by $(2^{j+2}j)^{2j^2}$. Passing to a subsequence if necessary, we may assume that for each fixed $j$, these blocks tend to limiting blocks as $k\to\infty$, and moreover, that the value of $A_{k,j}$ within each block also tends to a limit. This limit defines an asymmetric graphon; let us call it $W_j$. Clearly, $A_{k,j}\to W_j$ in the $L^1$ metric as $k\to\infty$.

Now note that by construction, the block structure for $W_{j+1}$ is a refinement of the block structure for $W_j$. Moreover, also by construction, the value of $W_j$ in one of its blocks is the average value of $W_{j+1}$ within that block. From this, by a standard martingale argument (for example, as in the proof of \cite[Theorem 9.23]{lovaszbook}) it follows that $W_j$ converges pointwise almost everywhere to an asymmetric graphon $W$ as $j\to\infty$. In particular, $W_j\to  W$ in $L^1$. We claim that $A_k\to W$ in the cut norm as $k\to\infty$. To show this, take any $\ve >0$. Find $j$ so large that $\|W-W_j\|_{L^1}\le \ve$ and  $2j^{-1}+6j^3 2^{-j}\le \ve$. Then for any $k$,
\begin{align*}
\|W-A_k\|_\Box &\le \|W-W_j\|_\Box + \|W_j - A_{k,j}\|_\Box + \|A_{k,j}-A_k\|_\Box\\
&\le \ve + \|W_j - A_{k,j}\|_{L^1} +  2j^{-1}+6j^3 2^{-j}\\
&\le 2\ve + \|W_j - A_{k,j}\|_{L^1}.
\end{align*}
Since $A_{k,j}\to W_j$ in $L^1$ as $k\to\infty$ and $\ve$ is arbitrary, this completes the proof. 
\end{proof}

\section*{Acknowledgments}
I thank Sohom Bhattacharya, Nicholas Cook, Terry Tao, and the anonymous referees for helpful comments.


\begin{thebibliography}{99}

\bibitem[Bhojanapalli and Jain(2014)]{bj14} {\sc Bhojanapalli, S.} and {\sc Jain, P.} (2014). Universal matrix completion. {\it Internat. Conf. Mach. Learn.,} 1881--1889.

\bibitem[Cand\`es and Plan(2010)]{candesplan10} {\sc Cand\`es, E.~J.} and  {\sc Plan, Y.} (2010). Matrix completion with noise. {\it Proc. of IEEE,} {\bf 98} no. 6, 925--936.


\bibitem[Cand\`es and Recht(2009)]{candesrecht09} {\sc Cand\`es, E.~J.} and {\sc Recht, B.} (2009). Exact matrix completion via convex optimization. {\it Found. Comput. Math.,} {\bf  9} no. 6, 717--772.

\bibitem[Cand\`es and Tao(2010)]{candestao10} {\sc Cand\`es, E.~J.} and {\sc  Tao, T.} (2010). The power of convex relaxation: Near-optimal matrix completion. {\it IEEE Trans. Inf. Theory,} {\bf  56} no. 5, 2053--2080.

\bibitem[Chatterjee(2015)]{chatterjee15} {\sc Chatterjee, S.} (2015). Matrix estimation by universal singular value thresholding. {\it Ann. Statist.,} {\bf  43} no. 1, 177--214.


\bibitem[Chen, Bhojanapalli, Sanghavi and Ward(2015)]{cbsw} {\sc Chen, Y., Bhojanapalli, S., Sanghavi, S.} and {\sc Ward, R.} (2015). Completing any low-rank matrix, provably. {\it J. Mach. Learn. Res.,} {\bf 16} no. 1, 2999--3034.

\bibitem[Eckart and Young(1936)]{ey36} {\sc Eckart, C.} and {\sc Young, G.}  (1936). The approximation of one matrix by another of lower rank. {\it Psychometrika,} {\bf 1} no. 3, 211--218.

\bibitem[Gavish and Donoho(2014)]{gd14} {\sc Gavish, M.} and {\sc Donoho, D.~L.} (2014). The optimal hard threshold for singular values is $4/\sqrt {3}$. {\it IEEE Trans. Inf. Theory,} {\bf 60} no. 8, 5040--5053.

\bibitem[Heiman, Schechtman and Shraibman(2014)]{hss14} {\sc Heiman, E., Schechtman, G.} and {\sc Shraibman, A.} (2014). Deterministic algorithms for matrix completion. {\it Random Structures Algorithms,} {\bf 45} no. 2, 306--317.

\bibitem[Johnson(1990)]{johnson90} {\sc Johnson, C.~R.} (1990). Matrix completion problems: A survey. In {\it Matrix theory and applications, Vol. 40, pp. 171--198.} Providence, RI.

\bibitem[Keshavan, Montanari and Oh(2010)]{kmo10} {\sc Keshavan, R.~H., Montanari, A.} and {\sc Oh, S.} (2010). Matrix completion from a few entries. {\it IEEE Trans. Inf. Theory,} {\bf 56} no. 6, 2980--2998.


\bibitem[Kir\'aly, Theran and Tomioka(2015)]{ktt15} {\sc Kir\'aly, F. J., Theran, L.} and {\sc Tomioka, R.} (2015). The algebraic combinatorial approach for low-rank matrix completion. {\it J. Mach. Learn. Res.,} {\bf 16} no. 1, 1391--1436.

\bibitem[Lee and Shraibman(2013)]{ls13} {\sc Lee, T.} and {\sc Shraibman, A.}  (2013). Matrix completion from any given set of observations. In {\it Adv. Neur. Inf. Proc. Sys.,} 1781--1787.

\bibitem[Lov\'asz(2012)]{lovaszbook} {\sc Lov\'asz, L. (2012).} {\it Large networks and graph limits.} American Mathematical Society, Providence, RI.

\bibitem[Nguyen, Kim and Shim(2019)]{nks19} {\sc Nguyen, L.~T., Kim, J.} and {\sc Shim, B.} (2019). Low-rank matrix completion: A contemporary survey. {\it IEEE Access,} {\bf 7}, 94215--94237.

\bibitem[Pimentel-Alarc\'on, Boston and Nowak(2016)]{pbn16} {\sc  Pimentel-Alarc\'on, D.~L., Boston, N.} and {\sc Nowak, R.~D.} (2016). A characterization of deterministic sampling patterns for low-rank matrix completion. {\it IEEE J. Selected Top. Sig. Proc.,} {\bf 10} no. 4, 623--636.

\bibitem[Singer and Cucuringu(2010)]{sc10} {\sc Singer, A.} and {\sc Cucuringu, M.} (2010). Uniqueness of low-rank matrix completion by rigidity theory. {\it SIAM J. Matrix Anal. Appl.,} {\bf  31} no. 4, 1621--1641.

\bibitem[Sportisse, Boyer and Josse(2018)]{sbj18} {\sc Sportisse, A., Boyer, C.} and {\sc Josse, J.} (2018). Imputation and low-rank estimation with Missing Non At Random data. {\it arXiv preprint arXiv:1812.11409.}

\end{thebibliography}
\end{document}